\documentclass[11pt]{amsart} 

\usepackage{amsmath}
\usepackage{amssymb}
\usepackage{graphicx}

\usepackage[margin=1.5in]{geometry}

\newcommand{\inv}[1]{\frac{1}{#1}}

\newcommand{\N}{\mathbb{N}}
\newcommand{\R}{\mathbb{R}}

\newcommand{\C}{\mathcal{C}}

\newcommand{\suml}[2]{\sum\limits_{#1}^{#2}}

\newcommand{\cupl}[2]{\bigcup\limits_{#1}^{#2}}

\newcommand{\proj}{\text{proj}}

\newcommand{\intd}[4]{\int\limits_{#1}^{#2} \! #3 \ \mathrm{d}#4}

\newcommand{\kein}{2^{d(n+i)}}

\newcommand{\ken}{2^{dn}}

\newcommand{\su}{\subset}

\DeclareMathOperator{\graph}{graph}

\theoremstyle{plain} \newtheorem{theorem}{Theorem}[section]
\theoremstyle{plain} \newtheorem{lemma}[theorem]{Lemma}
\theoremstyle{plain} 
\theoremstyle{plain} \newtheorem{cor}[theorem]{Corollary}
\theoremstyle{plain} 
\theoremstyle{plain} 
\theoremstyle{plain} 
\theoremstyle{plain} 
\theoremstyle{plain} \newtheorem{problem}[theorem]{Problem}
\theoremstyle{definition} \newtheorem{remark}[theorem]{Remark}
\theoremstyle{definition}
\theoremstyle{definition}

\title{Hausdorff dimension of union of lines that cover a curve}

\author[J. Cumberbatch]{James Cumberbatch}
\address[J. Cumberbatch]
{Department of Mathematics, Purdue University,
150 N. University Street, West Lafayette, IN 47907 USA}
\email{jcumberb@purdue.edu}

\author[T. Keleti]{Tam\'as Keleti}
\address[T. Keleti]
{Institute of Mathematics, E\"otv\"os Lor\'and University, P\'az\-m\'any P\'e\-ter s\'et\'any 1/c, H-1117 Budapest, Hungary}
\email{tamas.keleti@gmail.com}

\author[J. Zhang]{Jialin Zhang}
\address[J. Zhang]
{Department of Mathematics and Statistics,
Boston University, 111 Cummington Mall, Boston, MA 02215 USA}
\email{zjleric@bu.edu}

\thanks{This research was conducted as part of the Budapest Semesters in Mathematics Undergraduate Research Experience Program.
The second author was supported by the Hungarian National Research, Development and Innovation Office - NKFIH, 124749 and 129335.}



\begin{document}

\begin{abstract}
We construct a continuously differentiable curve in the plane that can be covered by a collection of lines
such that every line intersects the curve at a single point and the union of the lines has Hausdorff dimension
$1$.
We show that for twice-differentiable curves this is impossible. In that case, the union of the lines must have
Hausdorff dimension $2$. 
If we use only tangent lines then the differentiability of the curve already implies that the union of the
lines must have Hausdorff dimension $2$, unless the curve is a line.
We also construct a continuous curve, which is in fact the graph of a strictly convex function, such that
the union of (one-sided) tangent lines has Hausdorff dimension $1$.
\end{abstract}

\maketitle


\section{Introduction}
In 1952 R.~O.~Davies \cite{Da52} proved that any subset $A$ of the plane
can be covered by a collection of lines such that the union of the lines
has the same Lebesgue measure as $A$. Clearly, we cannot expect the same for Hausdorff dimension but for a given set $A$ of the plane or more generally 
of $\R^n$ it is natural to ask how small the Hausdorff dimension of a set can be that 
contains lines through every point of $A$. If $A$ itself is a line or union of lines then the answer is trivial but in this case we can require that we cover
$A$ with lines that intersect $A$ in small sets. So one can ask for example
the following two variants.

\begin{problem}\label{p:general}
Fix $A\su\R^n$. How small the Hausdorff dimension of $B$ can be if 
for every $x\in A$ the set $B$ contains a line $l$ through $x$ such that
(variant 1) $l \cap A$ is a singleton / (variant 2) $l\not\subset A$?
\end{problem}

If $A$ is an $(n-1)$-dimensional hyperplane then the two variants are clearly the same and in this case Problem~\ref{p:general} turns out to be closely related to the
Kakeya conjecture, which states that any Borel set in $\R^n$ that contains a unit line segment in every direction must have Hausdorff dimension $n$. 
Using a projective map one can easily show as in the proof 
of \cite[Theorem 1.2]{Ke} (see also \cite[Theorem 11.15]{Ma15}) 
that the Kakeya conjecture would imply that the answer to Problem~\ref{p:general} is $n$ when $A$ is a hyperplane, 
and conversely, if the answer to Problem~\ref{p:general} is $n$ for a hyperplane then every set that contains lines in every 
direction must have Hausdorff dimension $n$.
The Kakeya conjecture is open for $n\ge 3$, only partial results are known. For $n=3$ the best lower bound for the Hausdorff dimension is $5/2+\varepsilon$ for some small constant $\varepsilon>0$ by a recent result of Katz and Zahl \cite{KZ}, the best known lower bound for general $n$ is $(2-\sqrt 2)(n-4)+3$, proved in 2002 by Katz and Tao \cite{KT}.
Since in the plane Kakeya conjecture was proved by Davies \cite{Da71}, the above argument gives that
if $n=2$ and $A$ is a line then the answer to Problem~\ref{p:general} is $2$. 
Recently, L.~Venieri \cite[Theorem 9.2]{LV} proved the following more general result.
\begin{theorem}(Venieri)\label{t:venieri}
Let $f:\R^{n-1}\to\R$ be a 
Lipschitz function and let $A$ be a subset of the graph of
$f$ with positive $n-1$-dimensional Hausdorff measure. 
Let $B\su\R^n$ be a set such that for every $x\in A$ the set $B$ contains
a line $\ell_x$ through $x$, 
and at those $x$, where $f$ has a tangent hyperplane $T_x$, $\ell_x\not\subset T_x$.
Then the Hausdorff dimension of $B$ is at least $\frac{n+2}{2}$.
\end{theorem}
In fact, Venieri proved an even more general result 
for rectifiable sets and approximate tangent planes
and in her result instead of $l\su B$ it is enough to assume that
$B\cap l$ contains a line segment.
This is also related to the so called Nikodym sets,
which are Lebesuge measure zero subsets of $\R^n$
such that through every point of $\R^n$ there exists a line through $x$ whose intersection with the set contains a unit segment. 
The Hausdorff dimension of the Nikodym sets are also conjectured to be $n$ and this conjecture would follow from the Kakeya conjecture by the above mentioned argument.

In this paper we mainly work in the plane
($n=2$). 
Then Theorem~\ref{t:venieri} gives the sharp estimate $\dim B=2$ but it is
not clear what happens if we also allow tangent lines. 
One might expect that still we should get that the dimension of B be 2 unless some individual lines
cover too big a part of $A$. This will turn out to be wrong.
First we study the case when we use only tangent lines. 
We construct (Theorem~\ref{Tbinc}) a strictly convex function $F:[0,1]\to\R$
such that the union of its tangent lines and one-sided tangent lines has Hausdorff
dimension $1$. Note that strict convexity implies that every tangent line
or one-sided tangent line intersects the graph of $F$ in a single point. 
Recall that although a convex function does not have to be differentiable
everywhere, it is differentiable almost everywhere and the one-sided 
derivatives exist everywhere (if we also allow $\pm\infty$), so the one-sided
tangent lines exist everywhere, thus the tangent lines and one-sided tangent
lines cover the whole graph of $F$. 
We also prove (Theorem~\ref{TC1}) that there exists no such construction
if we require differentiability everywhere: the union of the tangent lines
of a differentiable function has always Hausdorff dimension $2$ unless the
function is linear.

And what if we allow both tangent and non-tangent lines? 
We show (Theorem~\ref{TCantc}) that in this case even continuous 
differentiability is not enough to exclude very small covers:
we construct a strictly convex
$\C^1$ function such that its graph 
can be covered with lines such that the union of the lines
has Hausdorff dimension $1$. Finally, we give necessary conditions about the
smoothness of the curve that guarantee that the union of the covering lines
has Hausdorff dimension $2$.

In this paper, for any given set $S$, we will use $H^d(S)$ to denote the $d$-dimensional Hausdorff measure of the set $S$, and we will use $\dim(S)$ to denote the Hausdorff dimension of $S$. 


\section{The results}
In the following lemma, the phrase ``tangent lines'' refers to left tangent lines, right tangent lines, or any line with slope between the two. The only thing required is for any ``tangent line'' to lie below the curve other than at its point of intersection. This criterion is fulfilled by regular tangent lines because of convexity. 
\begin{lemma}\label{lslope}
Let $F: [0,1]\to\R$ be a strictly convex function. Let $X$ be a set of lines tangent to $F$ in the above mentioned sense. 
Then the set $\bigcup X$ has Hausdorff dimension one if the set of the slopes of $X$ has Hausdorff dimension zero.
\end{lemma}
\begin{proof}
Let $Y$ be the code set of $X$, that is $Y=\{(a,b)\mid (y=ax+b) \in X\}$. Let $(a_1,b_1), (a_2,b_2)\in Y$. 
Our first goal is to show that $|b_1-b_2|\le|a_1-a_2|$.

Let $l_1(x) = a_1x+b_1$, and $l_2(x) = a_2x+b_2$. At the tangent points we have $a_ix_i+b_i=F(x_i)$, for $i\in\{1,2\}$.  
We can clearly suppose that $x_1\le x_2$.
We claim that there exists a point $x^*\in[x_1,x_2]\subset[0,1]$
such that $l_1(x^*)=l_2(x^*)$.
This clearly holds if $x_1=x_2$, so we can suppose that $x_1\neq x_2$.
Since $F$ is strictly convex, we have $l_1(x)<F(x)$ for all $x\neq x_1$. Thus, we have $l_1(x_2)< F(x_2) = l_2(x_2)$. Similarly we obtain $l_1(x_1) = F(x_1)>l_2(x_1)$. By continuity of $l_1-l_2$, 
we get that indeed
there must be a point $x^*\in [x_1,x_2]\subseteq[0,1]$ such that $l_1(x^*)=l_2(x^*)$ and so have $|b_1-b_2| = x^*|a_1-a_2| \leq |a_1 - a_2|$.

Thus $Y$ is the graph of a Lipschitz function defined on a set of Hausdorff dimension zero. Then $Y$ has Hausdorff dimension zero. Let $g:Y\times\R \to \R^2$ be the function $g((a,b),t)=(t,at+b)$. Note that
$\dim(Y\times\R)=\dim Y +1=1$, 
$g$ is locally Lipschitz and maps $Y\times\R$ to $\bigcup X$. Thus, $\bigcup X$ has dimension one.
\end{proof}

The following lemma is a special case of a result of Molter and Rela \cite{Furstenberg}
and also of a different result of Falconer and Mattila
\cite{KFPM}.

\begin{lemma}\label{MR}
Let $X$ be a set of lines in the plane. Then, the set $\bigcup X$ has Hausdorff dimension $2$ if the set of slopes of $X$ has Hausdorff dimension at least $1$. 
\end{lemma}

Our first goal is to show there exists a strictly convex function such that
the union of the set of one-sided tangent lines has dimension $1$.
The idea is the following. We construct a strictly increasing function whose image is zero dimensional. Then, its integral together with the set of one-sided tangent lines will satisfy Lemma \ref{lslope}. Thus, the union of the set of one-sided tangent lines will have Hausdorff dimension one. 

\begin{theorem}\label{Tbinc}
There exists a strictly convex function $F:[0,1]\rightarrow \R$ such that the union of the set of 
one-sided tangent lines $E$ has Hausdorff dimension $1$.
\end{theorem}

\begin{proof}
For all $x\in[0,1]$, we can write 
$x=\sum_{i=1}^\infty \omega_i/2^i,\text{ for }\omega_i\in\{0,1\} \text{ for all } i\in\N.$ Let $B=\{k/2^n\mid k,n\in \N; k\leq 2^n\}$. Observe that points in $B$ can be represented by two different sums (one is a finite sum, the other is an infinite sum), we shall always choose the finite sum. Consider the function $f:[0,1]\rightarrow [0,1]$, defined as
\[f\bigg(\sum_{i=1}^\infty \frac{\omega_i}{2^i}\bigg)=\sum_{i=1}^\infty\frac{\omega_i}{2^{i^2}}, \qquad \text{for any }\sum_{i=1}^\infty \frac{\omega_i}{2^i}\in[0,1].\]
Since monotonically increasing functions are integrable, we have that $f$ is integrable. Let $F:[0,1]\rightarrow \R$ be defined as
\[F(x)=\int_{0}^xf(x)\ dx.\]
Since $f$ is strictly increasing, we have that $F$ is a strictly convex function. 
Notice that $f$ is continuous on $[0,1]\setminus B$, so $F$ is differentiable on $[0,1]\setminus B$ and $F'(x)=f(x)$ on 
$[0,1]\setminus B$. 
Let $E$ be the set of one-sided tangent lines of $F$.

Second we are going to show that the union of $E$ has Hausdorff dimension $1$. 
Since $B$ is countable it is enough to consider the tangent lines at the points of $[0,1]\setminus B$.
Notice that the set of slopes of these tangent lines is exactly 
$f([0,1]\setminus B)$, so it is a subset of the range of $f$.
However, the range of $f$ can be covered by $2^n$ closed intervals with diameter $2^{-n^2}$. Then, for all $d>0$, we have
\[H^d\big(f([0,1])\big)\leq \lim\limits_{n\rightarrow \infty}2^n\bigg(\frac{1}{2^{n^2}}\bigg)^d= \lim\limits_{n\rightarrow \infty}\frac{1}{2^{d\cdot n^2-n}}=0.\] Then, we have $\dim\big(f([0,1])\big)=0$. By 
Lemma~\ref{lslope}, since $F$ is strictly convex, we 
get that the union of the one-sided tangent lines at the points 
of $[0,1]\setminus B$ has Hausdorff dimension $1$. Since 
$B$ is countable, this completes the proof.
\end{proof}

In Theorem \ref{Tbinc}, the function $F$ is differentiable almost everywhere. The following result shows that there is no such differentiable construction. 

\begin{theorem}\label{TC1}
Let $F:[0,1]\rightarrow \R$ be a differentiable function with nonconstant derivative. Then the union of the tangent lines of $F$ has Hausdorff dimension $2$.
\end{theorem}

\begin{proof}
Since $F'$ is non-constant, there exist $a,b\in [0,1]$ such that $F'(a)\neq F'(b)$. Without loss of generality, suppose $F'(a)<F'(b)$. By Darboux's Theorem, for all $c\in [F'(a),F'(b)]$, there exists $x\in [a,b]$ (or $x\in [b,a]$ if $a>b$) such that $F'(x)=c$. Since the derivative is the directions of the tangent lines, 
we obtained that
the set of directions of tangent lines has dimension $1$. 
Then, by
Lemma \ref{MR}, the union of the tangent lines of $F$ has Hausdorff dimension $2$.
\end{proof}

In Theorems \ref{Tbinc} and \ref{TC1}, we only used tangent lines (or one-sided tangents). Differentiable almost everywhere is the nicest condition of a such function whose union of tangent lines (one-sided tangent lines) has Hausdorff dimension $1$. The following theorem shows that if we use some non-tangent lines we can make a $\C^1$ function which fits the criteria.
\begin{theorem}\label{TCantc}
There exists a strictly convex $\C^1$ function $F:[0,1]\to\R$ and a set of lines $E$ such that $\graph(f) \subset \bigcup E$,
every line of $E$ intersects $\graph(f)$ in a single point
and $\dim\bigcup E = 1$.
\end{theorem}
\begin{proof}
%
%
Let $C$ be a zero Hausdorff-dimensional Cantor set, that is, a zero Hausdorff-dimensional nonempty perfect nowhere dense set in the interval $[0,1]$. 
First we construct a Cantor type function $\phi:\R\to[0,1]$ from $C$. 
Let $\mathcal J$ be the collection of the connected components of $\R\setminus C$ and let $D$ be the set of finite binary numbers in $[0,1]$. 
It is well known and not hard to show that there exists an order preserving bijection $\psi:\mathcal J\to D$, where we consider the natural ordering on $\mathcal J$.
On $\R\setminus C$ we define $\phi$ as
$\phi(x)=\psi(J)$ for $x\in J\in \mathcal J$.
The same way as in the case of the classical Cantor function one can easily show that $\phi$ can be extended to $C$ such that $\phi$ is continuous and non-decreasing.

Note that $\phi$ maps all points other than a set of dimension 0 to a set of dimension 0. We wish to construct a strictly increasing function which retains this property, the integral of which will yield the desired function.
To do so, first we will construct a sequence of functions $(f_n)_{n\in\N}$ inductively. In addition to $(f_n)_{n\in\N}$, $a_{n,i}$ and $b_{n,i}$ will be sequences of real numbers, and $C_n$ will be a sequence of 
compact
sets. All will be constructed in parallel. Let $C_1 = C$, and $f_1=\phi$. Suppose $C_n$ and $f_n$ are given. 
For each $n$ let $(a_{n,i},b_{n,i})_{i\in\N}$ be an enumeration of the maximal
intervals in $[0,1]\setminus C_n$.
It is not hard to check that we can also guarantee that \begin{equation}\label{nintcon}(a_{n,i_1},b_{n,i_1})\subset (a_{n-1,i_2},b_{n-1,i_2})\implies i_1\geq i_2.\end{equation} Let $C_{n+1} = C_n \cup \left(\cupl{i\in\N}{}  (b_{n,i}-a_{n,i})\cdot C+a_{n,i}\right)$. We define
$$f_{n+1}(x) = f_n(x)+\suml{i\in\N}{}\inv{2^{n+i+2}}\cdot \phi\left(\frac{x-a_{n,i}}{b_{n,i}-a_{n,i}}\right).$$
That is, we add a smaller copy of $\phi$ to each constant interval in $f_{n+1}$.
Let $f(x) = \lim\limits_{n\to\infty}f_n(x)$. As this is the uniform limit of continuous functions, it must also be continuous. 
Clearly, $f$ is non-decreasing.
Note that, by construction, for any $0\le x<y\le 1$ there exist $n$ and $i$ such that $x<a_{n,i}<b_{n,i}<y$, which implies 
that $f$ is strictly increasing.
Let $F(x)=\intd{0}{x}{f(t)}{t}$. 
Then $F'=f$, so $F$ is a strictly convex $\C^1$ function.
Let $C^* = \cupl{n\in\N}{} C_n$.

Now, we construct $E$, the set of lines which covers the graph of $F$. For each $x_0\in [0,1]$, there are two possibilities. If $x_0\in C^*$, add in the vertical line $x=x_0$ to the set $E$. If $x_0\not\in C^*$, add in the tangent line $f(x_0)(x-x_0)+F(x_0)$ to the set $E$. We claim that 
$\dim f([0,1]\setminus C^*)=0$.
By (\ref{nintcon}), for every $n,i\in\N$ we have that 
$$f(b_{n,i})-f(a_{n,i}) \le \suml{m\ge n}{}\ \ \suml{j:(a_{m,j},b_{m,j})\subset (a_{n,i},b_{n,i})}{}\inv{2^{m+j+2}} \le \suml{m\ge n}{}\suml{j\ge i}{}\inv{2^{m+j+2}} \le  \inv{2^{n+i}}.$$
Thus for any given $n$ we can cover the set 
$f([0,1]\setminus C^*)$
using $[f(a_{n,i}),f(a_{n,i})+\inv{2^{n+i}}]$ for all $i\in\N$. 
Hence for any $d>0$ we have
$$H^d\big(f([0,1]\setminus C^*)\big) \leq \lim_{n\to\infty}\suml{i\in\N}{}\inv{\kein} = \lim_{n\to\infty}\inv{\ken}\suml{i\in\N}{}\inv{2^{di}}=0.$$  Thus, we obtained $\dim f([0,1]\setminus C^*) = 0$. 
Note that the set of slopes of the tangent lines in $E$ is $f([0,1]\setminus C^*)$. Therefore,
by Lemma~\ref{lslope}, the Hausdorff dimension of the union of the set of tangent lines is $1$. The union of the set of non-tangent lines, being $C^*\times \R$, also has dimension $1$
since $\dim C^*=0$, so we can conclude that $\dim (\bigcup E) = 1$.
\end{proof}
The following results show that the constructed function in Theorem~\ref{TCantc} cannot be very smooth. 

\begin{theorem}\label{TC2}
Let $F$ be a differentiable real-valued Lipschitz function on 
$I=[0,1]$
with a derivative which maps sets of 
full Lebesgue outer measure in $I$
to sets of Hausdorff dimension one. Then the union of any set of lines which covers the graph of $F$ has Hausdorff dimension $2$.
\end{theorem}
\begin{proof}
Let $T$ be the set of $x$-coordinates of the points on the graph of $F$ covered by tangent lines. Let $U$ be the set of $x$-coordinates of the points covered by only non-tangent lines. There are two cases. 

Case 1: Suppose 
$T$ has full Lebesgue outer measure in $I$. 
Let $M$ be the set of slopes of the tangent lines
at the points of $T$.
By definition, we have that $M$ is $F'(T)$. Since $F'$ maps sets of 
full outer
measure to sets of Hausdorff dimension $1$, the set of slopes of the lines has Hausdorff dimension $1$. By Lemma~\ref{MR}, we thus have that the union of lines has Hausdorff dimension two.

Case $2$: Suppose 
that $T$ does not have full outer measure in $I$.
Then $U$ contains a Borel set $U'$ with positive Lebesgue measure.
Let $A = \graph F\!\restriction_{U'}$.  
Then $A$ has positive $1$-dimensional Hausdorff measure.
Since $A$ is a subset of the graph of a Lipschitz function,
its Hausdorff dimension clearly cannot exceed 1, so we have $\dim A=1$.
In this case, Theorem~\ref{t:venieri} tells us that the Hausdorff dimension of the union of the set of lines is $(2+2)/2 = 2$.
\end{proof}

Recall that a function is said to have the \emph{Luzin N property}
if it maps sets of Lebesgue measure zero to sets of Lebegue measure zero.

\begin{cor}\label{c:Luzin}
Let $F$ be a non-linear differentiable real-valued Lipschitz function on 
$I=[0,1]$ such that $F'$ has the Luzin N property.
Then the union of any set of lines which covers the graph of $F$ has Hausdorff dimension $2$.
\end{cor}

\begin{proof}
By Theorem~\ref{TC2}, it is enough to check that for any full Lebesgue outer measure set
$A \subset I$, the set $F'(A)$ has Hausdorff dimension $1$. 
Let $f=F'$.
Let $H$ be an arbitrary Borel set such that $f(A)\subset H$.
Since $f$ is a derivative, it is a Borel function, so $f^{-1}(H)$ is a Borel set.
Using that $A$ has full Lebesgue outer measure in $I$ and $A\subset f^{-1}(H)\subset I$, we obtain that $f^{-1}(H)$ has full Lebesgue measure in $I$.
Therefore $f^{-1}(f(I)\setminus H)=I\setminus f^{-1}(H)$ has Lebesgue measure zero.
Since $f=F'$ satisfies Luzin N property, we have 
that $f(I)\setminus H$
has Lebesgue measure zero. 
Notice that $F$ is non-linear and $f=F'$ is non-constant, so Darboux's theorem
implies that $f(I)$ is a non-degenerate interval. 
Therefore $H$ has positive Lebesgue measure.
Since $H$ was an arbitrary Borel set that contains $F'(A)$, this implies that $F'(A)$
has Hausdorff dimension $1$.
\end{proof}

Since differentiable functions satisfy the \emph{Luzin N property} and twice-differentiable functions are locally Lipschitz, we get the following.

\begin{cor}\label{c:C2}
Let $F$ be a non-linear twice-differentiable real-valued function on $[0,1]$. 
The union of any set of lines which covers the graph of $F$ has Hausdorff dimension $2$.
\end{cor}


We can extend this result to curves in the plane, and even in higher dimension.
Recall that a parametrized differentiable curve is called \emph{regular} if its derivative is nowhere zero.

\begin{cor}\label{c:curve}
Let $n\ge 2$, $I=[0,1]$  and let $\gamma:I\rightarrow \R^n$ be a twice-differentiable regular curve  
such that $\gamma(I)$ is not a subset of a line. Then, the union of any set of lines which covers $\gamma(I)$ has Hausdorff dimension at least $2$.
\end{cor}
\begin{proof}
We prove by induction. For $n=2$, this follows from Corollary~\ref{c:C2}. Suppose that 
$n\ge 3$ and 
the result holds for $n-1$. 
The goal is to project to an $n-1$-dimensional subspace and to apply the induction hypothesis.

First we prove the existence of a 
(non-degenerate) interval $J$ and non-parallel unit vectors $v, v'\in S^{n-1}$ such that 
$v$ is not parallel to any value of $\gamma'$ on $J$, $\proj_v \gamma(J)$ is not a subset of a line,
where $\proj_v$ is the orthogonal projection to the complementary subspace of $v$,
and for any two distinct $t_1,t_2\in J$ the points $\gamma(t_1)$ and $\gamma(t_2)$ are distinct, and the angles between the vectors $\gamma(t_2)-\gamma(t_1)$ and $v$, and between $\gamma(t_2)-\gamma(t_1)$ and $v'$ are in between $\pi/4$ and $3\pi/4$.
If 
there exists an interval $J\subset I$ such that $\gamma(J)$ is not a subset of a line but it
is contained in an $n-1$-dimensional affine subspace then the normal
vector of this affine subspace clearly works
as $v$ and any vector with angle $\pi/5$ with $v$ works as $v'$.
Otherwise, 
using that $\gamma'/|\gamma'|$ is continuous but not constant, 
we can choose an interval $J$ and a $u\in S^{n-1}$ such that the angle between $u$ and $\gamma'(t)$ is less than $\pi/5$ for every $t\in J$.
Then any non-parallel $v, v'\in S^{n-1}$ that are orthogonal to $u$ work.

Note that since $\gamma$ is a twice differentiable regular curve and $v$ is not parallel to any value of $\gamma'$ on $J$, the projection $\proj_v \gamma$ is
a twice differentiable regular curve on $J$.
If none of the given lines that cover $\gamma(J)$
are parallel to $v$ then projecting the covering lines to the complementary space of $v$, the induction hypothesis implies that the projection of the union of the covering lines has Hausdorff dimension at least $2$, which then implies that union of the original covering lines has also Hausdorff dimension at least $2$.

Therefore, to complete the proof it is enough to show that without loss of generality we can suppose that none of the covering lines of $\gamma(J)$ are parallel to $v$.
Since 
the direction between any two distinct points of $\gamma(J)$ is separated from $v$
and $v'$, the projections $\proj_v$ and $\proj_{v'}$
to the complementary spaces of $v$ and $v'$ are Lipschitz on $\gamma(J)$.
Note that this implies that for 
$u=v$ or $v'$
and any subset $E\subset \gamma(J)$ if $L$ is the union of lines parallel to $u$ through the points of $E$ then $\dim L=\dim (\proj_u E) + 1=\dim E + 1$ (where $\dim$ denotes Hausdorff dimension).
Let $E$ be the set of those points $p$ of $\gamma(J)$ that are covered by lines $\ell_p$ parallel to $v$. 
If $E$ has Hausdorff dimension at least $1$
then the union the lines $\ell_p$ through the points of $E$ already has Hausdorff dimension at least $2$ so we are done. 
Otherwise, 
replace each line $\ell_p$ through $p\in E$ by the line $\ell'_p$ through $p$ parallel to $v'$. 
Since $E$ has Hausdorff dimension less than $1$
the union of the lines $\ell'_p$ has Hausdorff dimension less than $2$, so by replacing each line $\ell_p$ by $\ell'_p$ we can indeed suppose that $E$ is empty, which completes the proof.
\end{proof}

\begin{remark}
In Corollary~\ref{c:curve} when $n>2$ we cannot claim that the union of the lines has Hausdorff dimension exactly $2$ even if every point of the curve is covered only by one line: let 
$\gamma:[a,b] \to\R^n$ be an arbitrary curve that is contained in a two-dimensional plane $P\su\R^n$,
let $Q\su\R^n$ be a cube disjoint to $P$
and let $g:[a,b]\to Q$ be a continuous onto map (an $n$-dimensional Peano curve). 
Then if for every $t\in[a,b]$ we take the line through $\gamma(t)$ and $g(t)$ then every point of the
curve $\gamma$ is covered exactly once but the union of the lines covers $Q$, so it has Hausdorff dimension $n$. 
\end{remark}

\begin{remark}
Note that in Corollary~\ref{c:curve} the assumption that the curve is not contained in a line guarantees that no covering line can be contained in the curve, so here
we consider variant 2 of Problem~\ref{p:general}.
\end{remark}

\begin{remark}
In Theorems~\ref{TC1} and \ref{TC2} and Corollaries~\ref{c:Luzin}, \ref{c:C2} and \ref{c:curve} if instead
of full lines we take only an arbitrary $1$ Hausdorff dimensional 
subset of each line then the union of these sets must also have  
Hausdorff dimension $2$ (at least $2$ in Corollary~\ref{c:curve}).
This follows immeditely from the result of H\'era, M\'ath\'e and
the first author \cite{HKM}, which states that if we have any collection
of lines in $\R^n$ such that the union of the lines has 
Hausdorff dimension at most $2$ and we take a $1$ Hausdorff dimensional
subset of each line then their union has the same Hausdorff dimension
as the union of the full lines.

Another way to get these slightly stronger results is the following.
All of the above-mentioned results are based on Lemma~\ref{MR}.
If we state the stronger version of this lemma in which we take
only the union of $1$ Hausdorff dimensional subsets of the lines, 
which is still a special case of the result of Molter and Rela~\cite{Furstenberg}, then our arguments give the above-mentioned stronger results.
\end{remark}

\section*{Acknowledgements}
We acknowledge the two anonymous referees for their constructive criticisms, which considerably improved the quality of the manuscript.
We are also grateful to Hanwen Liu for pointing out a mistake in an earlier version of this paper and for giving an idea how to fix that.


\end{document}